\documentclass[12pt,reqno]{amsart}
\usepackage{amsmath, amsthm, amssymb}

\advance \topmargin by -\headheight
\advance \topmargin by -\headsep
     
\evensidemargin \oddsidemargin
\newtheorem{theorem}{Theorem}[section]
\newtheorem{lemma}[theorem]{Lemma}
\newtheorem{corollary}[theorem]{Corollary}

\theoremstyle{definition}

\theoremstyle{remark}

\numberwithin{equation}{section}

  \def\bff{{\mathbf f}}

\def\bfh{{\mathbf h}}

\def\bfr{{\mathbf r}}

\def\calD{{\mathcal D}}

\def\calG{{\mathcal G}}

\def\dbC{{\mathbb C}}\def\dbN{{\mathbb N}}\def\dbP{{\mathbb P}}

\def\dbQ{{\mathbb Q}}

\def\Kbar{{\overline K}}\def\dbK{{\mathbb K}}

\def\grp{{\mathfrak p}}

\def\gam{{\gamma}} 
\def\del{{\delta}}

\def\ome{{\omega}} \def\Ome{{\Omega}}

\def\le{\leqslant} \def\ge{\geqslant}

\begin{document}
\title[Solvable points]{Solvable points on smooth projective varieties}
\author[Trevor D. Wooley]{Trevor D. Wooley}
\address{School of Mathematics, University of Bristol, University Walk, Clifton, Bristol BS8 1TW, United 
Kingdom}
\email{matdw@bristol.ac.uk}
\date{}
\begin{abstract} We establish that smooth, geometrically integral projective varieties of 
small degree are not pointless in suitable solvable extensions of their field of definition, 
provided that this field is algebraic over $\dbQ$.\end{abstract}
\maketitle

\section{Introduction} Given a field $K$ of characteristic $0$, consider the compositum 
$K^{\rm sol}$ of all solvable extensions of $K$. It was shown by Abel in 1823 that 
polynomials of degree $5$ or more in a single variable need not have their roots defined 
over $K^{\rm sol}$. There has been recent speculation that perhaps $K^{\rm sol}$ is so 
large that any geometrically irreducible projective curve defined over $K$ should possess a 
point defined in $K^{\rm sol}$ (see, for example, P\'al \cite[Question 1.2]{Pal2004} and 
Wooley \cite[page 63]{Woo1999}). Such has been confirmed by P\'al 
\cite[Theorem 1.6]{Pal2004} for smooth curves of genus $0$, $2$, $3$ and $4$. Much 
progress has also been made towards confirmation of this conjecture for curves of genus 
$1$ in the case $K=\dbQ$ by \c{C}iperiani and Wiles \cite{CW2008}. The situation, 
however, remains unclear both for curves of higher genus and higher dimensional varieties. 
On the one hand, P\'al \cite[Theorem 1.5]{Pal2004} has shown that when $g\ge 40$, there 
are local fields $F$ for which there exists a curve of genus $g$ failing to possess any point 
defined over $F^{\rm sol}$. On the other hand, P\'al \cite[Theorem 1.7]{Pal2004} has 
proved that any smooth, geometrically rational projective surface possesses a point defined 
over a solvable extension of its field of definition. In the absence of a more definitive 
resolution of this conjecture concerning solvable points on curves, and its analogue for 
surfaces, one is naturally led to enquire whether varieties of larger dimension might be 
guaranteed to possess solvable points. In this note, we establish that smooth, geometrically 
integral projective varieties are not pointless in solvable extensions of their field of 
definition, assumed algebraic over $\dbQ$, whenever their dimension is large enough in 
terms of their degree.\par

Our conclusions are in principle rather more general than the previous paragraph might 
suggest, though this observation hinges on the {\it Lefschetz principle}. It suffices here 
to describe the latter as asserting that any reasonable statement in algebraic geometry 
true over $\dbC$ is also true over {\it any} algebraically closed field of characteristic $0$. 
It seems fair to comment that there remains considerable uncertainty concerning the 
extent to which such a statement is true, or indeed makes sense (see Eklof 
\cite{Ekl1973} and Seidenberg \cite{Sei1958}). Thus, with safety in mind, we will 
restrict our conclusions to algebraic extensions of $\dbQ$, noting the potential for 
extension to arbitrary fields of characteristic $0$.

\begin{theorem}\label{theorem1.1} Let $X\subseteq \dbP^n$ be a smooth and 
geometrically integral variety defined over a field $K$ algebraic over $\dbQ$. Then 
$X$ possesses a point defined over a solvable extension of $K$ provided only that 
$\dim(X)\ge 2^{2^{\deg(X)}}$.
\end{theorem}

The lower bound constraint on the dimension is certainly large, and it is worth noting 
that improvement is certainly possible, especially for smaller degrees. However, when
 $\deg(X)$ is large, it seems that our methods are incapable of reducing this 
constraint in Theorem \ref{theorem1.1} to one of the shape 
$\dim(X)\ge 2^{2^{c\deg(X)}}$, for any $c<1$.\par

Our strategy for proving this theorem is simple. As has recently been observed by Browning 
and Heath-Brown \cite{BHB2014}, it follows from work of Bertram, Ein and Lazarsfeld 
\cite{BEL1991} that a (complex) smooth variety of dimension large enough in terms of 
its degree is automatically a complete intersection, and moreover its annihilating ideal is 
generated by forms defined over its field of definition. Having modified this strategy to the 
context of $K^{\rm sol}$ in \S2, we apply a diagonalisation method in \S3, based on that 
due to Brauer \cite{Bra1945}, to show that this complete intersection has a point defined 
in a solvable extension of the groundfield. The key input from the solubility of diagonal 
equations here is the trivial observation that, when $a_0,a_1\in K^\times$, then the 
diagonal equation $a_0x_0^d+a_1x_1^d=0$ possesses a solution in which $x_0$ and 
$x_1$ both lie in a solvable extension of $K$, namely $K(\sqrt[d]{-a_1/a_0})$. It follows, in 
fact, that the point lying on $X$ derived in Theorem \ref{theorem1.1} lies in a solvable field 
extension of $K$ defined by taking a tower of field extensions, each of degree no larger 
than $\deg(X)$. The elementary nature of our argument ensures that generalisations are 
easily obtained, and we mention a few in \S4.

\section{Passage to a complete intersection} We begin by adapting the treatment of 
Browning and Heath-Brown \cite{BHB2014} so that the groundfield is no longer restricted 
to be $\dbQ$. Throughout this section, we assume $K$ to be an algebraic extension of $\dbQ$. 
Let $X\subseteq \dbP^n$ be a smooth and geometrically integral variety 
defined over $K$. From Harris \cite[Corollary 18.12]{Har1995}, 
this variety lies in a linear subspace of dimension at most $\dim(X)+\deg(X)-1$. By 
considering the action of $\text{Gal}(\Kbar \colon K)$ on this linear space, it is apparent that 
there is no loss of generality in supposing it to be defined over $K$. We may therefore 
suppose without loss that
\begin{equation}\label{2.1}
n+1\le \dim(X)+\deg(X).
\end{equation}
Next, Bertram, Ein and Lazarsfeld \cite[Corollary 3]{BEL1991} show that whenever 
$X\subseteq \dbP^n$ is smooth and $\deg(X)\le \tfrac{1}{2}n/(n-\dim (X))$, then 
$X$ is a complete intersection. Then, as in \cite[\S1]{BHB2014}, we deduce from 
(\ref{2.1}) that such is the case whenever
$$\dim(X)>\deg(X)(2\deg(X)-3).$$

\par The argument of \cite[Lemma 3.3]{BHB2014} adapts to give the following conclusion.

\begin{lemma}\label{lemma2.1} Let $X\subseteq \dbP^n$ be a smooth complete 
intersection of codimension $R$ which is globally defined over $K$. Then there exist forms 
$F_1,\ldots ,F_R\in K[x_0,\ldots ,x_n]$ such that the annihilating ideal of $X$ is generated 
by $\{F_1,\ldots ,F_R\}$.
\end{lemma}

\begin{proof} This conclusion is immediate from \cite[Lemma 3.3]{BHB2014} in the case 
$K=\dbQ$, and the argument of its proof applies {\it mutatis mutandis} to deliver the more 
general conclusion recorded here. In essence, one applies the action of 
$\text{Gal}(\Kbar \colon K)$ to push the field of definition of the coefficients of elements of 
the annihilating ideal of $X$ down to the ground field by applying the natural trace operator.
\end{proof}

Finally, we record an immediate generalisation of a lemma presented in \cite{BHB2014}.

\begin{lemma}\label{lemma2.2} Let $\{F_1,\ldots ,F_R\}\subseteq K[x_0,\ldots ,x_n]$ be a 
a non-singular system of forms defining a variety $X$ in $\dbP^n$. Then the annihilating 
ideal of $X$ is generated by $\{F_1,\ldots ,F_R\}$, and $X$ is a smooth complete 
intersection of codimension $R$. In addition, the variety $X$ is geometrically integral, and 
has degree
$$\deg(X)=\deg(F_1)\cdots \deg(F_R).$$
\end{lemma}

\begin{proof} The desired conclusion is established in \cite[Lemma 3.2]{BHB2014} when the ground field 
is $\dbQ$. The argument of the latter proof applies, {\it mutatis mutandis}, in the present setting.
\end{proof}

We are now equipped to derive the principal conclusion of this section.

\begin{lemma}\label{lemma2.3} Let $X\subseteq \dbP^n$ be a smooth and geometrically 
integral variety, defined over a field $K$ algebraic over $\dbQ$, and satisfying the condition
\begin{equation}\label{2.2}
\dim(X)>2\deg(X)(2\deg(X)-3).
\end{equation}
Then there exist forms $F_1,\ldots ,F_R\in K[x_0,\ldots ,x_n]$ satisfying the conditions:
\item{(a)} $R=n-\dim(X)$;
\item{(b)} $\deg(X)=\deg(F_1)\cdots \deg(F_R)$;
\item{(c)} the point $(y_0\colon y_1\colon \ldots \colon y_n)\in \dbP^n$ lies on $X$ if and 
only if
$$F_j(y_0,\ldots ,y_n)=0\quad (1\le j\le R).$$
\end{lemma}

\begin{proof} Under the hypothesis (\ref{2.2}), it follows from Lemma \ref{lemma2.1} and 
its preamble that with $R=n-\dim(X)$, there exist forms 
$F_1,\ldots ,F_R\in K[x_0,\ldots ,x_n]$ such that the annihilating ideal of $X$ is generated 
by $\{F_1,\ldots ,F_R\}$. The claim (c) is an immediate consequence of the latter 
conclusion. Moreover, since $\{F_1,\ldots ,F_R\}$ must be a non-singular system of forms, 
it follows from Lemma \ref{lemma2.2} that $\deg(X)=\deg(F_1)\cdots \deg(F_R)$. This 
completes the proof of the lemma.
\end{proof}

\section{Brauer diagonalisation} We examine the existence of rational points on the 
complete intersection emerging from the previous section by means of a variant of the 
diagonalisation argument employed by Brauer \cite{Bra1945} in his work on Hilbert's 
resolvant problem. Let $\dbK$ be a field. Denote by $\calG_d^{(m)}(r_d,\ldots ,r_1)$ the 
set of $(r_d+\ldots +r_1)$-tuples of homogeneous polynomials, of which $r_i$ have degree 
$i$ for $1\le i\le d$, with coefficients in $\dbK$, possessing no non-trivial linear space of 
$\dbK$-rational solutions of projective dimension $m$. Define 
$V_d^{(m)}(\bfr)=V_d^{(m)}(r_d,\ldots ,r_1;\dbK)$ by putting
$$V_d^{(m)}(r_d,\ldots ,r_1;\dbK)=\sup_{\bfh \in \calG_d^{(m)}(r_d,\ldots ,r_1)}
\nu(\bfh),$$
in which $\nu(\bfh)$ denotes the number of variables appearing explicitly in $\bfh$. 
Likewise, denote by $\calD_{d,r}$ the set of $r$-tuples of diagonal polynomials of degree 
$d$, with coefficients in $\dbK$, which possess no non-trivial zeros over $\dbK$, and put
$$\phi_{d,r}(\dbK)=\sup_{\bff\in \calD_{d,r}}\nu(\bff).$$
Note that $V_d^{(m)}(r_d,\ldots ,r_1;\dbK)$ is an increasing function of the arguments 
$m$ and $r_d,\ldots,r_1$. We abbreviate $V_d^{(0)}(\bfr;\dbK)$ to $V_d(\bfr;\dbK)$, and 
$V_d(r,0,\ldots ,0;\dbK)$ to $v_{d,r}(\dbK)$. 
In additon, we abbreviate $\phi_{d,1}(\dbK)$ to $\phi_d(\dbK)$, and put
$$\psi_d(\dbK)=\sup_{1\le i\le d}\phi_i(\dbK).$$
We drop mention of $\dbK$ from all of these notations when the field of definition $\dbK$ 
is fixed. Note that whenever $n>\phi_d(\dbK)$, and $a_i\in \dbK$ $(0\le i\le n)$, then the 
equation $a_0x_0^d+\ldots +a_nx_n^d=0$ has a non-trivial solution over $\dbK$. A similar 
conclusion applies, concerning the existence of non-trivial linear spaces of solutions, 
regarding the notation $V_d^{(m)}(r_d,\ldots ,r_1;\dbK)$.\par

We first record \cite[equation (3.1)]{LS1983} in the form embodied in 
\cite[Lemma 2.3]{Woo1998a}.

\begin{lemma}\label{lemma3.1}
When $m$ is a positive integer, one has
$$V_d^{(m)}(r_d,\ldots ,r_1;\dbK)\le m+V_d(t_d,\ldots ,t_1;\dbK),$$
where
$$t_j=\sum_{i=j}^d r_im^{i-j}\quad (1\le j\le d).$$
\end{lemma}

Recall next the efficient diagonalisation procedure given in \cite[Lemma 2.2]{Woo1998a}.

\begin{lemma}\label{lemma3.2}
Let $d$ and $r_i$ $(1\le i\le d)$ be non-negative integers with $d\ge 2$ and $r_d>0$. 
Then whenever $\phi_d<\infty$ one has
$$V_d(r_d,\ldots,r_1;\dbK)\le r_d\phi_d+V_{d-1}(s_{d-1},\ldots ,s_1;\dbK),$$
where
$$s_j=\sum_{i=j}^dr_i(r_d\phi_d)^{i-j}\quad (1\le j\le d-1).$$
\end{lemma}

We ultimately apply Lemmata \ref{lemma3.1} and \ref{lemma3.2} only in situations 
wherein the arguments $r_i$ are distributed in a certain restricted manner. In order to 
facilitate the announcement of our key lemma, we describe a $d$-tuple $(r_d,\ldots ,r_1)$ 
as being {\it solid} when $r_i^2\ge r_{i-1}$ for $1<i\le d$.

\begin{lemma}\label{lemma3.3} Let $d$ and $r_d,\ldots ,r_1$ be positive integers with 
$d\ge 2$ satisfying the property that $(r_d,\ldots ,r_1)$ is solid. Then provided that 
$\psi_d(\dbK)<\infty$, one has
$$V_d(r_d,\ldots ,r_1;\dbK)\le 2r_d^{2^{d-1}}(\psi_d+1)^{2^{d-1}-1}.$$
\end{lemma}

\begin{proof} It follows from Lemma \ref{lemma3.2} that
$$V_d(r_d,\ldots ,r_1;\dbK)\le V_{d-1}(s_{d-1},\ldots ,s_2,s_1+r_d\phi_d;\dbK),$$
where
$$s_j=\sum_{i=j}^dr_i(r_d\phi_d)^{i-j}\quad (1\le j\le d-1).$$
Observe that for every integer $u$, one has $u\le 2^{u-1}$. Furthermore, when $u$ and 
$v$ are non-negative integers with $u\le v$, one has $v-u\le 2^v-2^u$. Thus, the 
hypothesis that $(r_d,\ldots ,r_1)$ is solid implies that for $1<j\le d-1$, one has
$$s_j\le \sum_{u=0}^{d-j}r_d^{2^u}(r_d\phi_d)^{d-j-u}\le \sum_{u=0}^{d-j}
r_d^{2^u}r_d^{2^{d-j}-2^u}\phi_d^{d-j-u},$$
whence
$$s_j\le r_d^{2^{d-j}}(\phi_d+1)^{d-j}\le r_d^{2^{d-j}}(\phi_d+1)^{2^{d-j-1}}.$$
Moreover, one sees in like manner that when $d\ge 3$, then
$$s_1+r_d\phi_d\le r_d\phi_d+r_d^{2^{d-1}}\sum_{u=0}^{d-1}\phi_d^{d-1-u}
\le r_d^{2^{d-1}}\sum_{u=0}^{d-1}\binom{d-1}{u}\phi_d^{d-1-u},$$
so that
$$s_1+r_d\phi_d\le r_d^{2^{d-1}}(\phi_d+1)^{d-1}\le 
r_d^{2^{d-1}}(\phi_d+1)^{2^{d-2}},$$
whilst, in the situation with $d=2$, one has
$$s_1+r_d\phi_d\le r_d\phi_d+r_d^{2^{d-1}}(\phi_d+1)\le 
2r_d^{2^{d-1}}(\phi_d+1)^{2^{d-2}}.$$
Consequently, whenever $d\ge 2$ and $(r_d,\ldots ,r_1)$ is solid, then
\begin{equation}\label{3.1}
V_d(r_d,\ldots ,r_1;\dbK)\le V_{d-1}(\ome_{d-1},
\ome_{d-1}^2,\ldots ,\ome_{d-1}^{2^{d-3}},\del \ome_{d-1}^{2^{d-2}};\dbK),
\end{equation}
where $\ome_{d-1}=r_d^2(\phi_d+1)$ and
$$\del=\begin{cases} 1,&\text{when $d\ge 3$,}\\
2,&\text{when $d=2$.}\end{cases}$$

\par The relation (\ref{3.1}) may be applied inductively to show that for each integer $u$ 
with $1\le u\le d-2$, one has
\begin{equation}\label{3.2}
V_d(r_d,\ldots ,r_1;\dbK)\le V_{d-u}(\ome_{d-u},\ome_{d-u}^2,\ldots 
,\ome_{d-u}^{2^{d-u-1}};\dbK),
\end{equation}
where for each $j$ we write
$$\ome_j=r_d^{2^{d-j}}(\psi_d+1)^{2^{d-j}-1}.$$
This claimed relation follows from (\ref{3.1}) when $u=1$, providing the base of the 
induction. Let $U$ be an integer with $2\le U\le d-2$, and assume that (\ref{3.2}) 
holds for $1\le u<U$. Then we find from (\ref{3.1}) that
\begin{align*}
V_d(r_d,\ldots ,r_1;\dbK)&\le V_{d-U+1}(\ome_{d-U+1},\ome_{d-U+1}^2,
\ldots ,\ome_{d-U+1}^{2^{d-U}};\dbK)\\
&\le V_{d-U}(\Ome,\Ome^2,\ldots ,\Ome^{2^{d-U+1}};\dbK),
\end{align*}
where
$$\Ome=\ome_{d-U+1}^2(\phi_{d-U+1}+1)\le \left( r_d^{2^{U-1}}
(\psi_d+1)^{2^{U-1}-1}\right)^2(\psi_d+1)=\ome_{d-U}.$$
This confirms the inductive step, so that, in particular, one has
$$V_d(r_d,\ldots ,r_1;\dbK)\le V_2(\ome_2,\ome_2^2;\dbK).$$
From here, an additional application of (\ref{3.1}) delivers the bound
$$V_d(r_d,\ldots ,r_1;\dbK)\le V_1(2\ome_1;\dbK)=2\ome_1=
2r_d^{2^{d-1}}(\psi_d+1)^{2^{d-1}-1}.$$
This completes the proof of the lemma.
\end{proof}

We refine Lemma \ref{lemma3.3} when all implicit equations have the same degree.

\begin{lemma}\label{lemma3.4}
Let $d$ and $r$ be positive integers with $d\ge 2$. Then whenever $\psi_d(\dbK)<\infty$, 
one has
$$v_{d,r}(\dbK)\le r\phi_d+2(r^2\phi_d)^{2^{d-2}}(\psi_{d-1}+1)^{2^{d-2}-1}.$$
\end{lemma}

\begin{proof} It follows from Lemma \ref{lemma3.2} that
\begin{equation}\label{3.3}
v_{d,r}(\dbK)\le r\phi_d+V_{d-1}(s_{d-1},\ldots ,s_1;\dbK),
\end{equation}
where $s_j=r(r\phi_d)^{d-j}$ $(1\le j\le d-1)$. We therefore find that 
$$s_j\le (r^2\phi_d)^{2^{d-j-1}}\quad (1\le j\le d-1).$$
But $(s_{d-1},\ldots ,s_1)$ is solid, and hence Lemma \ref{lemma3.3} delivers the bound
$$V_{d-1}(s_{d-1},\ldots ,s_1;\dbK)\le 2s_{d-1}^{2^{d-2}}(\psi_{d-1}+1)^{2^{d-2}-1}.
$$
The conclusion of the lemma is now immediate from (\ref{3.3}).
\end{proof}

Finally, we combine Lemmata \ref{lemma3.1} and \ref{lemma3.3} to provide a conclusion 
of use in investigating the existence of linear spaces of solutions.

\begin{lemma}\label{lemma3.5}
Let $d$ and $r_d,\ldots ,r_1$ be positive integers with $d\ge 2$ satisfying the property that 
$(r_d,\ldots ,r_1)$ is solid. Then provided that $\psi_d(\dbK)<\infty$, one has
$$V_d^{(m)}(r_d,\ldots ,r_1;\dbK)\le m+2\left( r_d^2(m+1)\right)^{2^{d-2}}
(\psi_d+1)^{2^{d-1}-1}.$$
\end{lemma}

\begin{proof} It follows from Lemma \ref{lemma3.1} that
\begin{equation}\label{3.4}
V_d^{(m)}(r_d,\ldots ,r_1;\dbK)\le m+V_d(t_d,\ldots ,t_1;\dbK),
\end{equation}
where
$$t_j=\sum_{i=j}^dr_im^{i-j}\quad (1\le j\le d).$$
The hypothesis that $(r_d,\ldots ,r_1)$ is solid implies that for $1\le j\le d$, one has
$$t_j\le \sum_{u=0}^{d-j}r_d^{2^u}m^{d-j-u}\le r_d^{2^{d-j}}(m+1)^{d-j}
\le \left( r_d^2(m+1)\right)^{2^{d-j-1}}.$$
The $d$-tuple $(t_d,\ldots ,t_1)$ is solid, and hence Lemma \ref{lemma3.3} shows that
$$V_d(t_d,\ldots ,t_1;\dbK)\le 2t_d^{2^{d-1}}(\psi_d+1)^{2^{d-1}-1}
\le 2\left(r_d^2(m+1)\right)^{2^{d-2}}(\psi_d+1)^{2^{d-1}-1}.$$
The conclusion of the lemma now follows from (\ref{3.4}).
\end{proof}

\section{The proof of Theorem \ref{theorem1.1}, and related conclusions} Theorem \ref{theorem1.1} 
follows from the case $m=0$ of the following theorem.

\begin{theorem}\label{theorem4.0} Let $X\subseteq \dbP^n$ be a smooth and 
geometrically integral variety defined over a field $K$ algebraic over $\dbQ$. Then 
$X$ possesses a projective linear space of dimension $m$ defined over a solvable extension 
of $K$ provided only that $\dim(X)\ge (2(m+1))^{2^{\deg(X)}}$.
\end{theorem}

\begin{proof} We begin with a straightforward 
consequence of Lemma \ref{lemma3.5}. Let $K$ be a field algebraic over $\dbQ$, and 
consider a natural number $j$. Given any elements $a_0,a_1\in K^{\rm sol}$, the equation 
$a_0x_0^j+a_1x_1^j=0$ possesses the non-trivial solution 
$(x_0,x_1)\in (\sqrt[j]{a_1},\sqrt[j]{-a_0})\in K^{\rm sol}\times K^{\rm sol}$, and thus
$\phi_j(K^{\rm sol})=1$. It follows that $\psi_d(K^{\rm sol})=1$ for every natural 
number $d$, and hence we deduce from Lemma \ref{lemma3.5} that whenever $d\ge 2$, 
and $(r_d,\ldots ,r_1)$ is solid, then
\begin{equation}\label{4.1}
V_d^{(m)}(r_d,\ldots ,r_1;K^{\rm sol})\le \left( r_d^2(m+1)\right)^{2^{d-2}}
2^{2^{d-1}}+m.
\end{equation}

\par Next, let $X\subseteq \dbP^n$ be a smooth and geometrically integral variety 
defined over $K$. Let $m$ be a non-negative integer, and 
put $N=\dim(X)$ and $D=\deg(X)$. Suppose that $N\ge \left( 2(m+1)\right)^{2^D}$. 
Then it follows from Lemma \ref{lemma2.3} that for some positive integer $R$, there exist 
forms $F_1,\ldots ,F_R\in K[x_0,\ldots ,x_n]$, with respective degrees $d_1,\ldots ,d_R$, 
satisfying the property that $(y_0\colon \ldots \colon y_n)\in \dbP^n$ is a 
$K^{\rm sol}$-rational point on $X$ if and only if $F_j(y_0,\ldots ,y_n)=0$ $(1\le j\le R)$. 
Moreover, we may suppose that
$$n-R\ge \left( 2(m+1)\right)^{2^D}\quad \text{and}\quad D=d_1\cdots d_R.$$
For $1\le j\le D$, put $r_j=\text{card}\{ 1\le i\le R:d_i=j\}$. Then provided that 
$n\ge V_D^{(m)}(r_D,\ldots ,r_1;K^{\rm sol})$, we see that $X$ possesses a 
$K^{\rm sol}$-rational point. The conclusion of the theorem therefore follows on 
confirming that
\begin{equation}\label{4.2}
V_D^{(m)}(r_D,\ldots ,r_1;K^{\rm sol})\le R+\left( 2(m+1)\right)^{2^D}.
\end{equation}

\par We divide into three cases, that in which $r_1=R$, a second in which $r_1<R$ and 
$r_D\ge 1$, and the final case with $r_1<R$ and $r_D=0$.\par

Suppose first that $r_1=R$, in which case $(r_D,\ldots ,r_1)=(0,\ldots ,0,R)$. The trivial 
relation $V_D^{(m)}(0,\ldots ,0,R;K^{\rm sol})=R+m$, that is a consequence of linear 
algebra, then delivers (\ref{4.2}) at once.\par

Next, when $r_1<R$ and $r_D\ge 1$, it follows from the relation $D=d_1\cdots d_R$ that 
$(r_d,\ldots ,r_1)$ takes the shape $(1,0,\ldots, 0,r_1)$ with $d\ge 2$. In such 
circumstances, one finds from (\ref{4.1}) that
\begin{align*}
V_D^{(m)}(r_D,\ldots ,r_1;K^{\rm sol})&\le r_1+V_D^{(m)}(1,0,\ldots ,0;K^{\rm sol})\\
&\le R+(m+1)^{2^{D-2}}2^{2^{D-1}}+m,
\end{align*}
and the desired upper bound (4.2) again follows.\par

Finally, suppose that $r_1<R$ and $r_D=0$. Here, sharper bounds than (\ref{4.2}) are in 
fact available, though we are challenged by issues of complexity. Let
$$d=\max \{ 1\le j\le D: r_j>0\}\quad \text{and}\quad r=\max\{r_j:2\le j\le D\}.$$
Then the relation $D=d_1\cdots d_R$ ensures that
$$2\le d\le \min\{D/2,D/2^{r-1}\}\quad \text{and}\quad 2^r\le D.$$
We now find from (\ref{4.1}) that
\begin{align}
V_d^{(m)}(r_d,\ldots ,r_1;K^{\rm sol})&\le r_1+V_d^{(m)}(r,r,\ldots ,r;K^{\rm sol})\notag 
\\
&\le R+\left( r^2(m+1)\right)^{2^{d-2}}2^{2^{d-1}}+m.\label{4.3} 
\end{align}
But
\begin{equation}\label{4.4}
r^{2^{d-1}}\le 2^{2^{d-1}r}\le 2^{2^{d-1}\cdot 2^{r-1}}\le 2^{2^{\nu -1}}
\end{equation}
where
$$\nu=r+d-1\le r-1+\min\{D/2, D/2^{r-1}\}.$$
One has $u+D/2^u\ge u+1+D/2^{u+1}$ whenever $D\ge 2^{u+1}$, so that since 
$2^r\le D$, we discern that $\nu \le 1+D/2$. On substituting this bound into (\ref{4.4}), 
and thence into (\ref{4.3}), we deduce that
\begin{align*}
V_d^{(m)}(r_d,\ldots ,r_1;K^{\rm sol})&\le R+(m+1)^{2^{D-2}}2^{2^{D/2}}\cdot 
2^{2^{D/2-1}}+m\\
&\le R+(m+1)^{2^D}2^{2^D-1}+m.
\end{align*}
The desired bound (\ref{4.2}) consequently follows in this final case.\par

Having confirmed the bound (\ref{4.2}) in all cases, we conclude that $X$ contains a 
$K^{\rm sol}$-rational linear space of projective dimension $m$. This completes the proof 
of the theorem, and hence also of Theorem \ref{theorem1.1}.  
\end{proof}

We mention in passing two further conclusions that may be proved in a manner very 
similar to the proof of Theorem \ref{theorem4.0}.

\begin{theorem}\label{theorem4.1}
Suppose that $p$ is a rational prime, and let $X\subseteq \dbP^n$ be a smooth and 
geometrically integral variety defined over $\dbQ$. Then $X$ possesses a point defined 
over $\dbQ_p$ provided only that $\dim(X)\ge \deg(X)^{2^{\deg(X)}}$.
\end{theorem}

\begin{proof} It follows from Davenport and Lewis \cite[Theorem 1]{DL1963} that 
$\phi_d(\dbQ_p)\le d^2$ for each natural number $d$. With $X\subseteq \dbP^n$ 
satisfying the hypotheses of the statement of the theorem, we put $N=\dim (X)$ and 
$D=\deg(X)$. Suppose that $N\ge D^{2^D}$. Then, as in the proof of Theorem 
\ref{theorem4.0}, it follows that $X$ is a complete intersection defined over $\dbQ$, 
and further that the conclusion of the theorem follows provided we are able to 
establish the bound
\begin{equation}\label{4.5}
V_D(r_D,\ldots ,r_1;\dbQ_p)\le R+D^{2^D},
\end{equation}
for all $D$-tuples $(r_D,\ldots ,r_1)$ with $R=r_D+\ldots +r_1$ satisfying  
$$D=D^{r_D}(D-1)^{r_{D-1}}\cdots 2^{r_2}.$$

\par When $r_1=R$, the bound (\ref{4.5}) follows via linear algebra. Also, when $r_1<R$ 
and $r_D\ge 1$, one finds as before that $(r_D,\ldots ,r_1)=(1,0,\ldots ,0,r_1)$. In such 
circumstances, an application of Lemma \ref{lemma3.4} gives
\begin{align*}
V_D(r_D,\ldots ,r_1;\dbQ_p)&\le r_1+D^2+2(D^2)^{2^{D-2}}
\left( (D-1)^2+1\right)^{2^{D-2}-1}\\
&\le R+2D^{2^D-1}\le R+D^{2^D},
\end{align*}
confirming (\ref{4.5}). Finally, when $r_1<R$ and $r_D=0$, a treatment akin to that 
applied in the proof of Theorem \ref{theorem4.0} conveys us from Lemma \ref{lemma3.3} 
to the bound
$$V_D(r_D,\ldots ,r_1;\dbQ_p)\le R+D^{2^{D/2+1}-1}\le R+D^{2^D},$$
again confirming (\ref{4.5}). Thus $X$ does indeed possess a $\dbQ_p$-rational point. 
This completes the proof of the theorem.
\end{proof}

\begin{corollary}\label{corollary4.2}
Let $X\subseteq \dbP^n$ be a smooth and geometrically integral variety defined over 
$\dbQ$. Then $X$ possesses a point defined over $\dbQ$ provided only that it possesses 
a real point and $\dim(X)\ge \deg(X)^{2^{\deg(X)}}$. In particular, when $\deg (X)$ is 
odd and the latter condition on the dimension is satisfied, then $X$ possesses a point 
defined over $\dbQ$.
\end{corollary}

\begin{proof} It follows from Browning and Heath-Brown \cite[Theorem 1.1]{BHB2014} that $X$ 
satisfies the Hasse Principle provided only that $\dim (X)\ge (\deg(X)-1)2^{\deg(X)}-1$. Since 
$d^{2^d}>(d-1)2^d-1$ for $d\ge 1$, the first conclusion is immediate from Theorem \ref{theorem4.1}. 
When $\deg(X)$ is odd, moreover, it follows from Lemma \ref{lemma2.3}(b) that $X$ is a 
complete intersection of hypersurfaces of odd degree. In such circumstances, it follows 
that $X$ possesses a real point (this follows as a consequence of the Borsuk-Ulam 
Theorem, or see as an alternative \cite[Theorem 15]{Lan1953}), and hence the desired 
result follows from the first conclusion of the theorem.
\end{proof}

\begin{theorem}\label{theorem4.3} Suppose that $p$ is a rational prime, and that $K$ is an algebraic 
extension of $\dbQ_p$. Let $L$ be an algebraic extension of $\dbQ$ embedding into $K$, and let 
$X\subseteq \dbP^n$ be a smooth and geometrically integral variety defined over $L$.
Then $X$ possesses a point defined over $K$ provided only that
$$\dim(X)\ge \exp\left( 2^{\deg(X)+2}\left( \log \deg (X)\right)^2\right) .$$
\end{theorem} 

\begin{proof} The annihilating ideal of $X$ is defined by polynomials having coefficients 
in some finite field extension $L_0$ of $\dbQ$. It follows from Brink, Godinho and 
Rodrigues \cite[Theorem 1]{BGR2008} that when $d=p^\tau m$ with $p\nmid m$, and 
$K_0$ is any field extension of $\dbQ_p$ of finite degree, then one has 
$\phi_d(K_0)\le d^{2\tau+5}$. It follows that $\phi_d(K)\le d^{2\tau+5}$. The former conclusion 
improves on an earlier result of Skinner \cite{Ski2006} (correcting \cite{Ski1996}). Observe here that 
$\tau\le (\log d)/(\log 2)$. A modicum of computation reveals that when $d\ge 3$, one has 
$2[(\log d)/(\log 2)]+5<8\log d$, and thus $\phi_d(K)<\exp\left( 8(\log d)^2\right)$. Note also the 
classical result (in the special case $d=2$ relevant for quadratic forms) to the effect that $\phi_2(K)=4$. 
Write $D$ for $\deg(X)$. Then, with these results in hand, one may follow the argument of the proof of 
Theorem \ref{theorem4.1}, {\it mutatis mutandis}, to show that 
$X$ possesses a $K$-rational point provided only that $\dim(X)$ exceeds
$$2\left( \exp \left( 8(\log D)^2\right) \right)^{2^{D-1}-1}\le \exp\left( 2^{D+2}(\log D)^2\right) .$$
The conclusion of the theorem now follows.
\end{proof}

We note that our earlier work \cite[Corollaries 1.2 and 1.3]{Woo1998a} addresses the 
existence of rational points on certain complete intersections, over field extensions of 
$\dbQ_p$, and over purely imaginary field extensions of $\dbQ$, respectively. The proofs of 
these corollaries, and also the proof of \cite[Theorem 10.13]{Woo1999}, use as input the 
main result of Skinner \cite{Ski1996}. The correction of the latter paper embodied in 
\cite{Ski2006}, and improved in \cite{BGR2008}, provides a substitute for the infelicitous work 
of \cite{Ski1996} that suffices to recover all of these conclusions, with one modification. Namely, 
the revised version of \cite[Corollary 1.3]{Woo1998a} shows that when 
$d\in \dbN$ and $L$ is a purely imaginary field extension of $\dbQ$, then 
$v_{d,r}(L)\le r^{2^{d-1}}e^{2^{d+1}d}$ (acquiring a factor of $2$ in the exponent of $e$ 
relative to the original statement). We note, in particular, that when $d$ is odd and $K$ is a field extension 
of $\dbQ_p$, then the proof of \cite[Theorem 1]{BGR2008} shows that $\phi_d(K)\le d^{2\tau+3}$, 
where $d=p^\tau m$ with $p\nmid m$ (note in the penultimate line of that 
paper that $\gam=\tau+1$ when $p\ne 2$). Thus, when $d$ is odd, one has $\phi_d(K)\le e^{2d}$, and 
the argument of the proof of \cite[Theorem 10.13]{Woo1999} proceeds without further 
modification.

\begin{corollary}\label{corollary4.4}
Let $L$ be an algebraic extension of $\dbQ$. Then $X$ possesses a point defined over $L$ 
provided only that $X$ has a point defined over all completions of $L$ at the infinite place, and
 in addition
$$\dim(X)\ge \exp\left( 2^{\deg(X)+2}(\log \deg(X))^2\right) .$$
In particular, should this condition on $\dim(X)$ be satisfied, then $X$ possesses an $L$-rational point 
when $L$ is purely imaginary, and also when $\deg(X)$ is odd.
\end{corollary}

\begin{proof} Frei and Madritsch \cite[Theorem 1.4]{FM2014} show that $X$ satisfies the 
Hasse principle provided only that $\dim(X)\ge (\deg(X)-1)2^{\deg(X)}-1$. But when $d\ge 1$, one has 
$\exp(2^{d+2}(\log d)^2)>(d-1)2^d-1$, and thus the first conclusion is immediate from Theorem 
\ref{theorem4.3}. When $L$ is purely imaginary, it is immediate that $X$ has a point defined over all 
completions of $L$ at the infinite place, since this amounts to possessing a point over $\dbC$. This 
confirms the second assertion of the theorem. When $\deg(X)$ is odd, meanwhile, it follows 
from Lemma \ref{lemma2.3}(b) that $X$ is a complete intersection of hypersurfaces of odd 
degree. In such circumstances, it follows as before that $X$ possesses a real point, and hence 
the claimed result follows from the first conclusion of the theorem.
\end{proof}

In view of \cite[Remark 2.8.2]{CW2008}, it may be of interest to restrict attention to totally 
real solvable extensions of $\dbQ$. Motivated by such considerations, we are able to 
derive as a special case of Corollary \ref{corollary4.4} a conclusion which avoids working in 
any extension of the groundfield whatsoever.

\begin{corollary}\label{corollary4.5} Let $X\subseteq \dbP^n$ be a smooth and 
geometrically integral variety of odd degree defined over a totally real field $K$. Then $X$ 
possesses a point defined over $K$ provided only that $\dim(X)\ge 
\exp\left( 2^{\deg(X)+2}(\log \deg(X))^2\right)$.
\end{corollary}

\bibliographystyle{amsbracket}

\begin{thebibliography}{18}

\bibitem{BEL1991}
A. Bertram, L. Ein and R. Lazarsfeld, \emph{Vanishing theorems, a theorem of Severi, and the 
equations defining projective varieties}, J. Amer. Math. Soc. \textbf{4} (1991), no. 3, 
587--602. 

\bibitem{Bra1945}
R. Brauer, \emph{A note on systems of homogeneous algebraic equations}, Bull. Amer. Math. 
Soc. \textbf{51} (1945), 749--755.

\bibitem{BGR2008}
D. Brink, H. Godinho and P. H. A. Rodrigues, \emph{Simultaneous diagonal equations over 
$\grp$-adic fields}, Acta Arith. \textbf{132} (2008), no. 4, 393--399.

\bibitem{BHB2014}
T. D. Browning and D. R. Heath-Brown, \emph{Forms in many variables and differing degrees}, 
preprint; arXiv:1403.5937.

\bibitem{CW2008}
M. \c{C}iperiani and A. Wiles, \emph{Solvable points on genus one curves}, Duke Math. J. \textbf{142} 
(2008), no. 3, 381--464.

\bibitem{DL1963}
H. Davenport and D. J. Lewis, \emph{Homogeneous additive equations}, Proc. Roy. Soc. 
Ser. A \textbf{274} (1963), 443--460.

\bibitem{Ekl1973}
P. C. Eklof, \emph{Lefschetz's principle and local functors}, Proc. Amer. Math. Soc. \textbf{37} (1973), 
333--339.

\bibitem{FM2014}
C. Frei and M. Madritsch, \emph{Forms of differing degrees over number fields}, preprint, 
arXiv:1412.6419.

\bibitem{Har1995}
J. Harris, \emph{Algebraic geometry. A first course}, Corrected reprint of the 1992 original, Graduate Texts 
in Mathematics, 133, Springer-Verlag, New York, 1995.

\bibitem{Lan1953}
S. Lang, \emph{The theory of real places}, Ann. of Math. (2) \textbf{57} (1953), 378--391.

\bibitem{LS1983}
D. B. Leep and W. M. Schmidt, \emph{Systems of homogeneous equations}, Invent. Math. 
\textbf{71} (1983), no. 3, 539--549.

\bibitem{Pal2004}
A. P\'al, \emph{Solvable points on projective algebraic curves}, Canad. J. Math. \textbf{56} 
(2004), no. 3, 612--637.

\bibitem{Sei1958}
A. Seidenberg, \emph{Comments on Lefschetz's principle}, Amer. Math. Monthly \textbf{65} (1958), 
685--690. 

\bibitem{Ski1996}
C. M. Skinner, \emph{Solvability of $\grp$-adic diagonal equations}, Acta Arith. \textbf{75} 
(1996), no. 3, 251--258.

\bibitem{Ski2006}
C. M. Skinner, \emph{Local solvability of diagonal equations (again)}, Acta Arith. \textbf{124} 
(2006), no. 1, 73--77.

\bibitem{Woo1998a}
T. D. Wooley, \emph{On the local solubility of Diophantine systems}, Compositio Math. 
\textbf{111} (1998), no. 2, 149--165.

\bibitem{Woo1999}
T. D. Wooley, \emph{Diophantine problems in many variables: the role of additive number 
theory}, Topics in number theory (University Park, PA, 1997), 49--83, Math. Appl., 
\textbf{467}, Kluwer Acad. Publ., Dordrecht, 1999.

\end{thebibliography}
\providecommand{\bysame}{\leavevmode\hbox to3em{\hrulefill}\thinspace}

\end{document}